\documentclass[a4paper, reqno, 10pt]{amsart}
\usepackage{amssymb,tikz,amsthm}
\usepackage{hyperref}

\numberwithin{equation}{section}
\begin{document}

\author{Hai Jin}
\address{School of Mathematical Sciences,\ Shanghai Jiao Tong University,  \ Shanghai 200240, \ China}
\curraddr{}
\email{jinhaifyh@sjtu.edu.cn}
\thanks{This work was supported by National Natural Science Foundation of China (Grant No. 12131015), and Natural
Science Foundation of Shanghai (Grant No. 23ZR1435100)}

\newtheorem{lemdef}{Lemma-Definition}[section]
\newtheorem{theorem}{Theorem}[section]
\newtheorem{corollary}[theorem]{Corollary}
\newtheorem{lemma}[theorem]{Lemma}
\newtheorem{proposition}[theorem]{Proposition}
\newtheorem{examples}[theorem]{Examples}
\newtheorem{fact}[theorem]{Fact}

\theoremstyle{definition}
\newtheorem{definition}[theorem]{Definition}
\newtheorem{remark}[theorem]{Remark}

\def\Id{{\rm Id}}
\def\rad{{\rm rad}}

\title{Graded automorphisms of quantum affine spaces}

\begin{abstract}
This paper computes the graded automorphism group of quantum affine spaces. Specifically, we determine that this group is isomorphic to a semi-direct product of a blocked diagonal matrix group and a permutation group.
\end{abstract}

\maketitle
\section{\bf Introduction}



Let ${\bf q}=(q_{ij})\in {\mathcal M}_n(k)$ satisfies $q_{ij}q_{ji}=1$ and $q_{ii}=1$ for $1\le i\le j\le n$. Recall that a multi-parameter quantum affine space \cite{BG2002}, or simply quantum affine space, is a $k$-algebra 
$$
   \mathcal{O}_{\bf q}(k^{n})= k\langle x_1,\cdots, x_n\rangle/\langle x_jx_i-q_{ij}x_ix_j, 1\le i,j\le n \rangle.
$$
When $n=2$, it is called quantum plane. This is an $\mathbb{N}$-graded algebra $ \mathcal{O}_{\bf q}(k^{n})=\oplus_{i=0}^{\infty}A_i$ with $A_1 = \langle x_1,\cdots, x_n \rangle$. Denote ${\rm Aut}(\mathcal{O}_{\bf q}(k^{n}))$ and ${\rm Aut_{gr}}( \mathcal{O}_{\bf q}(k^{n}))$ to be the automorphism group and $\mathbb{N}$-graded automorphism group of $\mathcal{O}_{\bf q}(k^{n})$.           

The automorphism groups of $\mathcal{O}_{\bf q}(k^n)$ were firstly described by Alev and Chamarie in \cite{AC1996}. Using derivations, they computed ${\rm Aut}(\mathcal{O}_{\bf q}(k^{n}))$ when $n$ and ${\bf q}$ are under some restrictions. For example, when $q_{ij}=q\in k^*$ for all $i<j$ and $q$ is not a root of unit, they showed
\[
{\rm Aut}(\mathcal{O}_{\bf q}(k^n))\cong
\left\{
    \begin{aligned}
        &(k^*)^3\rtimes k, &&n=3 \\
        &(k^*)^n, &&n\ne 3.
    \end{aligned}
\right.
\]
When $q=-1$($\ne 1$), Kirkman, Kuzmanovich and Zhang proved ${\rm  Aut}_{\rm gr}(\mathcal{O}_{\bf q}(k^n)) \cong (k^*)^n \rtimes S_n$ \cite{KKZ2014}.
Ceken, Palmieri, Wang and Zhang introduced a discriminant method in \cite{CPWZ2015}, which is very effective to compute the automorphism group of some non-commutative algebras and has many applications \cite{CPWZ2016,CYZ2018}. 
In particular, they showed when ${\bf q}$ satisfies some conditions, the automorphism group coincides with the graded automorphism group, i.e., ${\rm Aut}(\mathcal{O}_{\bf q}(k^n))={\rm Aut}_{\rm gr}(\mathcal{O}_{\bf q}(k^n))$ \cite[Theorem 3.4]{CPWZ2016}. Furthermore, they provided ${\rm Aut}_{\rm gr}(\mathcal{O}_{\bf q}(k^n))$ when $q_{ij}\ne 1$ for all $i\ne j$ \cite[Proposition 3.9(1)]{CPWZ2016}: 
\begin{equation*}
    {\rm Aut}_{\rm gr}(\mathcal{O}_{\bf q}(k^n))\cong (k^*)^n\rtimes \mathcal{P}_{\bf q},
\end{equation*}
where $\mathcal{P}_{\bf q}$ is a permutation group (see Definition \ref{defcp}). 
This constraint is relatively mild; however, it does exclude some interesting quantum affine spaces, e.g., the initial quantum cluster(that is, the quantum affine space generated by an initial cluster) attached to some uniparameter symmetric quantum nilpotent algebras, see examples in \cite{GY2017}, \cite{GY2020} and \cite{GY2021}. Therefore, our research seeks to compute ${\rm Aut}_{\rm gr}(\mathcal{O}_{\bf q}(k^n))$ beyond this limitation.

In this paper, we present the detailed structure of ${\rm Aut}_{\rm gr}(\mathcal{O}_{\bf q}(k^n))$. Specifically, we extend \cite[Proposition 3.9(1)]{CPWZ2016} by removing their assumptions on the matrix ${\bf q}$, thereby deriving the graded automorphism group for arbitrary quantum affine spaces. 
To achieve this, we firstly give a sufficient and necessary condition for a matrix ${\bf m}\in {\rm GL}_n(k)$ to induce a graded automorphism of $\mathcal{O}_{\bf q}(k^{n})$ (see Theorem \ref{theoremquantumplane2}), which strengthens Lemma 1.4.1 in \cite{AC1996}.
Then, we give the main result Theorem \ref{theoremmain}, which describes the graded automorphism group of $\mathcal{O}_{\bf q}(k^{n})$, i.e.,
\[
    {\rm  Aut}_{\rm gr}(\mathcal{O}_{\bf q}(k^{n}))\cong (\prod_{1\le i\le m}{\rm GL}_{|B_i|}(k)) \rtimes \mathcal{P}_{\bf q}/\mathcal{I}_{\bf q} 
\]
This theorem involves a partition $\{B_1,\cdots, B_m\}$ of the index set $\{1,\cdots, n \}$, which is introduced in \cite{KKZ2010} (see Definitions \ref{defblock}), and two permutation groups $\mathcal{P}_{\bf q}$ and $\mathcal{I}_{\bf q}$ (see Definition \ref{defcp} and \ref{defip}). 

Throughout this paper, let $k$ be the base field, $k^*= k\setminus \{0\}$, algebras and maps are over $k$. 

\section{\bf Preliminaries}


\subsection{\bf Graded automorphisms of quantum affine spaces}

If $f\in {\rm Aut_{gr}}(\mathcal{O}_{\bf q}(k^{n}))$, then there is ${\bf m}\ = (m_{ij})\in{\rm GL}_n(k)$ such that $f(x_i)=\sum\limits_{1\le j\le n}m_{ij}x_j, \ 1\le i\le n$. Consider the following representation 
$$
\rho: {\rm Aut_{gr}}(\mathcal{O}_{\bf q}(k^{n}))\rightarrow {\rm GL}_n(k), \ f\mapsto {\bf m}^T.
$$ 
This is a faithful representation, thus ${\rm Aut_{gr}}(\mathcal{O}_{\bf q}(k^{n}))$ is isomorphic to a subgroup of ${\rm GL}_n(k)$, which will denote by $M(\mathcal{O}_{\bf q}(k^{n}))$.
Conversely, given a matrix ${\bf m}=(m_{ij})\in {\rm GL}_n(k)$, one can define a $k$-linear map $f:A_1 \rightarrow A_1$ by $f(x_i)=\sum\limits_{1\le j\le n}m_{ij}x_j, \ 1\le i\le n$. This map induces a graded algebra automorphism of $A$ in the natural way if and only if ${\bf m}^T\in M(\mathcal{O}_{\bf q}(k^{n}))$. 

Alev and Chamarie gave a description of $M(\mathcal{O}_{\bf q}(k^{n}))$ when the characteristic of $k$ is zero \cite{AC1996}, which can be generalized to fields of arbitrary characteristic as follows:
\begin{lemma}\cite[Lemma 1.4.1]{AC1996}\label{lemmaquantumplane1}
    Suppose ${\bf m}=(m_{ij})\in {\rm GL}_n(k)$. Then ${\bf m}^T\in M(\mathcal{O}_{\bf q}(k^{n}))$ if and only if the following conditions are satisfied.
    \begin{align}
        \label{eqlemmaquantumplane11}
        &(q_{ij}-q_{st})m_{is}m_{jt}=(1-q_{ij}q_{st})m_{js}m_{it}; \\
        \label{eqlemmaquantumplane12}&(q_{ij}-1)m_{is}m_{js}=0, \ \forall \ 1\le i< j\le n,\ 1\le s< t\le n.
    \end{align}
\end{lemma}
\begin{proof}
    We only prove the necessity. Let $f$ be the graded automorphism induced by ${\bf m}^T$, i.e., $f$ satisfies $f(x_i)=\sum\limits_{1\le j\le n}m_{ij}x_j, \ 1\le i\le n$. Then $f(x_jx_i-q_{ij}x_ix_j) = 0 $ for $1\le i,j\le n$. Notice $f(x_ix_j)=f(x_i)f(x_j)$, thus for $1\le i< j\le n$ and $1\le s< t\le n$ one gets 
    \begin{align*}
        &(m_{js}x_s+m_{jt}x_t)(m_{is}x_s+m_{it}x_t) = q_{ij}(m_{is}x_s+m_{it}x_t)(m_{js}x_s+m_{jt}x_t).
    \end{align*}    
    By considering the coefficients of $x_s^2$ and $x_sx_t$ on both sides, this is equivalent to the Equations (\ref{eqlemmaquantumplane11}) and (\ref{eqlemmaquantumplane12}).
\end{proof}

\begin{remark}\label{remarkquantumplane1}
    It is easy to see that Equations $(\ref{eqlemmaquantumplane11})$ and $(\ref{eqlemmaquantumplane12})$ also hold when $i=j$. Actually, by interchanging the indices, one can observe that $(\ref{eqlemmaquantumplane11})$ and $(\ref{eqlemmaquantumplane12})$ is equivalent to the following equations:
    \begin{align*}
        &(q_{ij}-q_{st})m_{is}m_{jt}=(1-q_{ij}q_{st})m_{js}m_{it}, \\ 
        &(q_{ij}-1)m_{is}m_{js}=0, \ \forall \ 1\le i, j, s, t\le n.
    \end{align*}
\end{remark}


\subsection{\bf An improvement of Lemma \ref{lemmaquantumplane1}}

The following theorem gives a more specific description of $M(\mathcal{O}_{\bf q}(k^{n}))$ than Lemma \ref{lemmaquantumplane1}.
\begin{theorem}\label{theoremquantumplane2}
    Suppose ${\bf m}=(m_{ij})\in {\rm GL}_n(k)$. Then ${\bf m}^T\in M(\mathcal{O}_{\bf q}(k^{n}))$ if and only if $(q_{ij}-q_{st})m_{is}m_{jt}=0, \  \forall\ 1\le i,j,s,t\le n.$
\end{theorem}
\begin{proof}
    The sufficiency is provided by Lemma \ref{lemmaquantumplane1}. For necessity, by Remark \ref{remarkquantumplane1} one has 
    \begin{align}
        \label{eqlemmaquantumplane21}&(q_{ij}-q_{st})m_{is}m_{jt}=(1-q_{ij}q_{st})m_{js}m_{it}, \\
        \label{eqlemmaquantumplane22}&(q_{ij}-1)m_{is}m_{js}=0, \ \forall \ 1\le i, j\le n,\ 1\le s, t\le n. 
    \end{align}
    If the statement is not true, then there are some $i,j,s,t\in \{1,\cdots,n\}$ such that 
    \begin{equation}\label{eqlemmaquantumplane3}
        (q_{ij}-q_{st})m_{is}m_{jt}=(1-q_{ij}q_{st})m_{js}m_{it}\ne 0,
    \end{equation}
    thus $m_{is}m_{it}m_{js}m_{jt}\ne 0$. Moreover, one gets $s\ne t$ since by (\ref{eqlemmaquantumplane22}) one has $(q_{ij}-q_{st})m_{is}m_{jt}=(q_{ij}-1)m_{is}m_{jt}=0$. 
    
    Consider the following $2\times 2$ matrix:
    \begin{equation*}
        {\bf m}_{\{i,j\},\{s,t\}}=\begin{pmatrix}
            m_{is}&m_{it}\\
            m_{js}&m_{jt}
        \end{pmatrix}.
    \end{equation*}
    If $i= j$, then it is trivial that ${\rm Det}({\bf m}_{\{i,i\},\{s,t\}})=0$. If $i\ne j$, by (\ref{eqlemmaquantumplane22}) and $m_{is}m_{js}\ne 0$, one gets $q_{ij}=1$. Thus (\ref{eqlemmaquantumplane3}) becomes 
    \begin{equation*}
        (1-q_{st})m_{is}m_{jt}=(1-q_{st})m_{js}m_{it}\ne 0,
    \end{equation*}
    and hence $1-q_{st}\ne 0$ and ${\rm Det}({\bf m}_{\{i,j\},\{s,t\}})=0$. Therefore, in both cases one can conclude that ${\rm Det}({\bf m}_{\{i,j\},\{s,t\}})=0$. Since ${\bf m}$ is invertible, there exists $h\in \{1,\cdots,n\}-\{i\}$ such that ${\rm Det}({\bf m}_{\{h,i\},\{s,t\}})\ne 0$. 

    Now consider matrix ${\bf m}_{\{h,i\},\{s,t\}}$:
    \begin{equation*}
        {\bf m}_{\{h,i\},\{s,t\}}=
        \begin{pmatrix}
            m_{hs}&m_{ht}\\
            m_{is}&m_{it}
        \end{pmatrix}.
    \end{equation*}
    By Remark \ref{remarkquantumplane1} one has 
    \begin{align}
        \label{eqlemmaquantumplane31}&(q_{hi}-q_{st})m_{hs}m_{it}=(1-q_{hi}q_{st})m_{is}m_{ht}, \\
        \label{eqlemmaquantumplane32}&(q_{hi}-1)m_{hs}m_{is}=0, \\
        \label{eqlemmaquantumplane33}&(q_{hi}-1)m_{ht}m_{it}=0.  
    \end{align}
    Note that ${\rm Det}({\bf m}_{\{h,i\},\{s,t\}})\ne 0$, at least one of $\{m_{hs}, m_{ht} \}$ is nonzero. Since $m_{is}m_{it}\ne 0$, by (\ref{eqlemmaquantumplane32}) and (\ref{eqlemmaquantumplane33}) one gets $q_{hi}=1$. Therefore, (\ref{eqlemmaquantumplane31}) becomes
    \begin{equation*}
        (1-q_{st})m_{hs}m_{it}=(1-q_{st})m_{is}m_{ht}.
    \end{equation*}
    However, $1-q_{st}\ne 0$ indicates ${\rm Det}({\bf m}_{\{h,i\},\{s,t\}})=m_{hs}m_{it}-m_{is}m_{ht}= 0$, a contradiction. This completes the proof.
\end{proof}

From this theorem one can see that $M(\mathcal{O}_{\bf q}(k^{n}))$ is closed under matrix transpose.
\begin{corollary}\label{corollaryquantumplane}
    Suppose ${\bf m}=(m_{ij})\in {\rm GL}_n(k)$. Then the following statements are equivalent.\\
    ${\rm (1)}$ ${\bf m}^T\in M(\mathcal{O}_{\bf q}(k^{n}))$;\\
    ${\rm (2)}$ ${\bf m}\in M(\mathcal{O}_{\bf q}(k^{n}))$;\\
    ${\rm (3)}$ $(q_{ij}-q_{st})m_{is}m_{jt}=0, \  \forall\ 1\le i,j,s,t\le n.$
\end{corollary}

\section{\bf Compatible permutations and block decomposition}

\subsection{\bf Skeleton permutations}
Let $S_n$ be the symmetric group. The following lemma is by linear algebra, which can also be seen as a corollary of \cite[Lemma 7.1]{Ga2013}.
\begin{lemma}\label{skeleton}
    Suppose ${\bf m}=(m_{ij})\in {\rm GL}_n(k)$. Then there is $\pi\in S_n$ such that $m_{\pi(i)i}\ne 0$ for $1\le i\le n$. We will call this $\pi$ a skeleton permutation of ${\bf m}$, and denote the set of skeleton permutations of ${\bf m}$ by ${\rm Skel}({\bf m})$.
\end{lemma}
\begin{proof}
    Since ${\bf m}$ is invertible, the determinant of ${\bf m}$ is nonzero. One gets 
    \[
        {\rm Det}({\bf m}) = \sum_{\pi\in S_n}{\rm sgn}(\pi)\prod_{i=1}^{n}m_{\pi(i)i}\ne 0.
    \]
    Therefore, there is $\pi\in S_n$ such that $m_{\pi(i)i}\ne 0$ for $1\le i\le n$. 
\end{proof}
Denote the matrix representation of permutation $\pi$ by ${\bf r}_{\pi}$, i.e., ${\bf r}_\pi=(r_{ij})$ satisfies
\[
    r_{ij}=
    \left\{
        \begin{array}{l}
            1, \ \ \ i=\pi(j);\\
            0, \ \ \ i\ne \pi(j),
        \end{array} 
    \right.  
\]
then ${\rm Skel}({\bf r}_{\pi})=\{\pi\}$.


\subsection{\bf Compatible permutations}
To characterize ${\rm Aut_{gr}}(\mathcal{O}_{\bf q}(k^{n}))$, we need the definition of compatible permutations, which is closely related to the isomorphism problems of quantum affine spaces and their quotient algebras, see \cite{Ga2013}, \cite{BZ2017} and \cite{JZ2023}. 
\begin{definition}\label{defcp}
    A permutation $\pi\in S_n$ is {\it compatible} with $\mathcal{O}_{\bf q}(k^{n})$, or simply compatible, if $q_{\pi(i)\pi(j)}=q_{ij}$ for $1\le i,j\le n$.
    Denote the set of compatible permutations of $\mathcal{O}_{\bf q}(k^{n})$ by ${\mathcal P}_{\bf q}$, i.e.,
    \[
    \mathcal{P}_{\bf q}:=\{\pi\in S_n\ |\ q_{\pi(i)\pi(j)}=q_{ij}, \ \forall\ 1\le i,j\le n \}.
    \]
    Then $\mathcal{P}_{\bf q}$ is a subgroup of $S_n$.
\end{definition}

\begin{examples}
    Suppose $\mathcal{O}_{{\bf q}_1}(k^3)$ and $\mathcal{O}_{{\bf q}_2}(k^3)$ are two quantum affine spaces with  
    ${\bf q}_1 =
    \left(
        \begin{smallmatrix} 
            1 &-1 &a\\ 
            -1 &1&a\\
            \frac{1}{a}& \frac{1}{a} &1 
        \end{smallmatrix}
    \right) $,  
    ${\bf q}_2 = 
    \left(
        \begin{smallmatrix} 
            1           &a          & \frac{1}{a}\\ 
            \frac{1}{a} &1          &a\\
            a           &\frac{1}{a} &1 
        \end{smallmatrix}
    \right)$, $\pm 1\ne a\in k^*$.
    Then $\mathcal{P}_{\bf q_1}=\{(1), (1,2)\}$ and $\mathcal{P}_{\bf q_2}=\{(1), (1,2,3), (1,3,2)\}$.
\end{examples}
The following lemma demonstrates the connections between $\mathcal{P}_{\bf q}$ and ${\rm Aut}_{\rm gr}(\mathcal{O}_{\bf q}(k^{n}))$.
\begin{lemma}\label{lemcon} Suppose $\pi\in S_n$ and ${\bf m}\in {\rm GL}_n(k)$.\\
   ${\rm (1)}$ If $\pi\in\mathcal{P}_{\bf q}$, then ${\bf r}_{\pi}\in M(\mathcal{O}_{\bf q}(k^{n}))$.\\
   ${\rm (2)}$ If ${\bf m}\in M(\mathcal{O}_{\bf q}(k^{n}))$, then ${\rm Skel}({\bf m})\subset \mathcal{P}_{\bf q}$.
\end{lemma}
\begin{proof} 
    (1) For  $1\le i,j,s,t\le n $ one has 
    \[
        (q_{ij}-q_{st})r_{is}r_{jt}=(q_{ij}-q_{st})r_{is}r_{jt}\delta_{i,\pi(s)}\delta_{j,\pi(t)}, 
    \]
    which is nonzero only when $i=\pi(s)$ and $j=\pi(t)$. While in this case one gets
    \[
        (q_{ij}-q_{st})r_{is}r_{jt} = q_{\pi(s)\pi(t)}-q_{st}=0, \ 
    \]
    thus by Corollary \ref{corollaryquantumplane}, ${\bf r}_{\pi}\in M(\mathcal{O}_{\bf q}(k^{n}))$.\\
    (2) Suppose $\pi\in {\rm Skel}({\bf m})$. By Corollary \ref{corollaryquantumplane} one has $(q_{\pi(i)\pi(j)}-q_{ij})m_{\pi(i)i}m_{\pi(j)j}=0$ for $1\le i,j\le n$. By definition of ${\rm Skel}({\bf m})$, $m_{\pi(i)i}\ne 0$ for $1\le i\le n$, thus $q_{\pi(i)\pi(j)}-q_{ij}=0$ for $1\le i,j\le n$, i.e., $\pi\in\mathcal{P}_{\bf q}$.
    The proof is done.
\end{proof}

\subsection{\bf Block decomposition and invariant compatible permutations}

We use the block decomposition in \cite{KKZ2010}, which is a partition of $\{1,\cdots,n\}$ and is uniquely determined by ${\bf q}$.
\begin{definition}\cite[Definition 3.1]{KKZ2010}\label{defblock}
    Fix a ${\bf q}$ and $1\le i\le n$.\\
    $(1)$ We define the {\it block} containing $i$ to be
    \[
        B(i) = \{1\le i'\le n\ |\ q_{i'j} = q_{ij}, \ \forall \  1\le j \le n\}.
    \]
    We say that $i$ and $j$ are in the same block if $B(i) = B(j)$.\\
    $(2)$ We use the blocks to define an equivalence relation on $\{1,\cdots,n\}$, and then $\{1,\cdots,n\}$ is a disjoint union of blocks
    \[
        \{1,\cdots,n\} = \bigcup_{i\in W}B(i)
    \]
    for some index set $W$. The above equation is called the {\it block decomposition}.\\
    $(3)$ We denote by $m$ the number of blocks, i.e. $m :=| W |$. We will write $B(i)$ as $B_w$ for some $1\le w\le m$, thus $\{B_i\}_{1\le i\le m}$ is a partition of $\{1,\cdots,n\}$.
\end{definition}

\begin{remark}\label{remarkpar}
    Actually, $i$ and $j$ are in the same block if and only if the $i$-th row and the $j$-th row of ${\bf q}$ are the same.
\end{remark}



\begin{lemma}\label{lemmablock}
    For $1\le i,j\le n$, the following statements are equivalent.\\
   {\rm  (1)} $B(i)=B(j)$.\\
   {\rm  (2)} $B(\pi(i))=B(\pi(j))$ for all ${\pi}\in\mathcal{P}_{\bf q}$.\\
   {\rm  (3)} $B(\pi(i))=B(\pi(j))$ for some ${\pi}\in\mathcal{P}_{\bf q}$.
\end{lemma}
\begin{proof}
    $(1)\Rightarrow (2)$ Suppose that ${\pi}\in\mathcal{P}_{\bf q}$ and $B(i)=B(j)$. Then one has 
    \[
        q_{\pi(i)s} = q_{\pi(i) \pi(\pi^{-1}(s))}=q_{i\pi^{-1}(s)}=q_{j\pi^{-1}(s)}=q_{\pi(j)s}, \ \ \forall\ 1\le s\le n,
    \]
    thus by definition $B(\pi(i))=B(\pi(j))$.\\
    $(2)\Rightarrow (3)$ is trivial.\\
    $(3)\Rightarrow (1)$ Suppose that ${\pi}\in\mathcal{P}_{\bf q}$ such that $B(\pi(i))=B(\pi(j))$. Then
    \begin{align*}
        q_{is} = q_{\pi(i)\pi(s)}=q_{\pi(j)\pi(s)}=q_{js}, \ \ \forall\ 1\le s\le n,
    \end{align*}
    thus one gets $B(i)=B(j)$.
\end{proof}

\begin{corollary}\label{corpi}
    For each $\pi \in \mathcal{P}_{\bf q}$, define the action of $\pi$ on a subset $S\subset \{1,\cdots, n\}$ by $\pi(S) = \{\pi(s)\ | \ s\in S  \}$. Then for each $1\le i\le n$, $\pi(B(i))=B(\pi(i))$.
\end{corollary}
\begin{proof}
    Firstly, $\pi(B(i))$ is nonempty because $\pi(i)\in \pi(B(i))$. For each $s\in \pi(B(i))$, $\pi^{-1}(s)\in B(i)$, thus $B(\pi^{-1}(s))=B(i)$. By Lemma \ref{lemmablock}, one gets $B(s)=B(\pi(i))$, hence $\pi(B(i))\subset B(\pi(i))$. 

    On the other hand, for each $t\in B(\pi(i))$, by Lemma \ref{lemmablock} one has $B(\pi^{-1}(t))=B(i)$, thus $\pi^{-1}(t)\in B(i)$. By definition one gets $t\in \pi(B(i))$ and $B(\pi(i))\subset \pi(B(i))$, thus $\pi(B(i))=B(\pi(i))$. The proof is done.
\end{proof}


By this corollary, a compatible permutation $\pi\in\mathcal{P}_{\bf q}$ is also a permutation of the partition $\{B_i\}_{1\le i\le m}$. The following definition is natural.

\begin{definition}\label{defip}
    A permutation $\pi\in \mathcal{P}_{\bf q}$ is called {\it invariant}, if $\pi(B(i))=B(i)$ for $1\le i\le n$.
    Denote the set of invariant compatible permutations of $\mathcal{O}_{\bf q}(k^{n})$ by ${\mathcal I}_{\bf q}$, i.e.,
    $\mathcal{I}_{\bf q} = \{\pi\in \mathcal{P}_{\bf q}\ | \ \pi(B(i))=B(i),\ \forall \ 1\le i\le n\}$.
\end{definition}
\begin{lemma}\label{lemmaiq}
    $\mathcal{I}_{\bf q}\lhd \mathcal{P}_{\bf q}$. 
\end{lemma}
\begin{proof}
    Easy to verify that $\mathcal{I}_{\bf q}$ is a subgroup of $\mathcal{P}_{\bf q}$. By definition, for each $\pi \in \mathcal{I}_{\bf q}$ and $1\le i\le n$, one has $q_{\pi(i)j}=q_{ij}, \ \forall \ 1\le j\le n.$ 
    Then for any $\sigma\in \mathcal{P}_{\bf q}$,
    \begin{align*}
        q_{\sigma^{-1}\pi\sigma(i)j}=q_{\pi(\sigma(i))\sigma(j)}=q_{\sigma(i)\sigma(j)}=q_{ij},\ \forall \ 1\le j\le n,
    \end{align*}
    which means that $\sigma^{-1}\pi\sigma(B(i))=B(\sigma^{-1}\pi\sigma(i))=B(i)$, thus $\sigma^{-1}\pi\sigma\in \mathcal{I}_{\bf q}$. This proves that $\mathcal{I}_{\bf q}$ is a normal subgroup of $\mathcal{P}_{\bf q}$.
\end{proof}

\section{\bf Graded automorphism group of quantum affine spaces}



Let $\Phi$ be the group homomorphism defined as follows: 
\[
    \Phi:\prod\limits_{1\le i\le m}{\rm GL}_{|B_i|}(k)\rightarrow {\rm GL}_n(k),\ \  (M_1,\cdots, M_m)\mapsto {\bf m},
\] 
where ${\bf m}=(m_{ij})$ is defined by specifying its submatrices: ${\bf m}_{B_i,B_j}= \delta_{ij}M_i$ for $1\le i,j\le m$, where ${\bf m}_{B_i,B_j}$ is the submatrix of ${\bf m}$ with row index set $B_i$ and column index set $B_j$. 
$\Phi(M_1,\cdots, M_m)= {\bf m}$ is indeed invertible, as ${\bf m}$ is similar to a blocked diagonal matrix ${\rm diag}(M_1,\cdots, M_m)$, and each $M_i$ is invertible. Furthermore, $\Phi$ is injective, thus one has $\prod\limits_{1\le i\le m}{\rm GL}_{|B_i|}(k)\cong {\rm Im}\ \Phi$.
Additionally, one has the following lemma:

\begin{lemma}\label{lemmaimphi}
    Suppose ${\bf m}\in {\rm GL}_n(k)$. Then ${\bf m}\in {\rm Im}\ \Phi$ if and only if $m_{ij}=0$ for $B(i)\ne B(j)$, $1\le i,j\le n$.
\end{lemma}
\begin{proof}
    The necessity follows from the definition of $\Phi$. For sufficiency, notice that ${\bf m}$ is similar to ${\rm diag}({\bf m}_{B_1,B_1},\cdots, {\bf m}_{B_m,B_m})$, thus ${\rm Det}({\bf m})=\prod\limits_{1\le i\le m} {\rm Det}({\bf m}_{B_i,B_i})\ne 0$. In particular, ${\rm Det}({\bf m}_{B_i,B_i})\ne 0$ for all $1\le i\le m$. 
    Therefore, ${\bf m}_{B_i,B_i}\in {\rm GL}_{|B_i|}(k)$ for $1\le i\le m$. It follows directly that $\Phi({\bf m}_{B_1,B_1},\cdots, {\bf m}_{B_m,B_m})={\bf m}$. 
\end{proof}

To show ${\rm Im}\ \Phi$ is a normal subgroup of ${\rm Aut_{gr}}(\mathcal{O}_{\bf q}(k^{n}))$ (as stated in Lemma \ref{lemmaPhi2}), we require the following two lemmas.

\begin{lemma}\label{lemskel}
    Suppose that ${\bf m}=(m_{ij})\in M(\mathcal{O}_{\bf q}(k^{n}))$ and $\pi\in {\rm Skel}({\bf m})$.\\
  ${\rm (1)}$ If $m_{\pi(i)j}\ne 0$ for some $1\le i,j\le n$, then $B(i)=B(j)$.\\
  ${\rm (2)}$ ${\bf r}_{\pi}{\bf m}\in {\rm Im}\ \Phi$.
\end{lemma}
\begin{proof}
    (1) If $i=j$, then $B(i)=B(j)$. Now suppose $i\ne j$ and let $1 \le s \le n$.  
    By Corollary \ref{corollaryquantumplane}, one has 
    \[
        (q_{\pi(i)\pi(s)}-q_{js})m_{\pi(i)j}m_{\pi(s)s}=0.
    \]
    By the definition of ${\rm Skel}({\bf m})$ and the assumption $m_{\pi(i)j}\ne 0$, one gets $m_{\pi(i)j}m_{\pi(s)s}\ne 0$. Thus $q_{js}=q_{\pi(i)\pi(s)}=q_{is}$ for all $1\le s\le n$, i.e., $B(i)=B(j)$.\\
    (2) Fix some $1\le i,j\le n$ such that $B(i)\ne B(j)$. Note that $({\bf r}_{\pi}{\bf m})_{ij}=m_{\pi(i)j}$, then by (1) one gets $({\bf r}_{\pi}{\bf m})_{ij}=m_{\pi(i)j}=0$. By Lemma \ref{lemmaimphi}, ${\bf r}_{\pi}{\bf m}\in {\rm Im}\ \Phi$.
\end{proof}

\begin{lemma}\label{lemmaPhi0}
    Suppose that ${\bf m}\in {\rm Im}\Phi$ and $\pi\in\mathcal{P}_{\bf q}$. Then ${\bf r}_{\pi}{\bf m}{\bf r}_{\pi^{-1}}\in {\rm Im} \Phi$.
\end{lemma}
\begin{proof}
    Note that for $1\le i,j\le n$ and $\pi\in \mathcal{P}_{\bf q}$, ${\bf r}_{\pi}{\bf m}{\bf r}_{\pi^{-1}}\in {\rm GL}_n(k)$ and $({\bf r}_{\pi}{\bf m}{\bf r}_{\pi^{-1}})_{ij}=m_{\pi(i)\pi(j)}$. If $B(\pi(i))\ne B(\pi(j))$ for some $i,j$, then by Lemma \ref{lemmablock}, $B(i)\ne B(j)$, thus by Lemma \ref{lemmaimphi} one gets $m_{\pi(i)\pi(j)}=0$. Again by Lemma \ref{lemmaimphi}, ${\bf r}_{\pi}{\bf m}{\bf r}_{\pi^{-1}}\in {\rm Im} \Phi$.
\end{proof}

\begin{lemma}\label{lemmaPhi2}
    ${\rm Im} \Phi\lhd  M(\mathcal{O}_{\bf q}(k^{n}))$.
\end{lemma}
\begin{proof}
    Firstly we show ${\rm Im}\Phi$ is a subgroup of $ M(\mathcal{O}_{\bf q}(k^{n}))$. Suppose ${\bf m}\in {\rm Im} \Phi$. By Corollary \ref{corollaryquantumplane}, we need to compute $(q_{ij}-q_{st})m_{is}m_{jt}$ for $i,j,s,t \in \{1,\cdots, n\}$. Now fix some $i,j,s,t$. 
    By Lemma \ref{lemmaimphi}, $m_{is}\ne 0$ only when $B(i)=B(s)$, $m_{jt}\ne 0$ only when $B(j)=B(t)$. While in this case, then $q_{ij}= q_{sj}= q_{js}^{-1} = q_{ts}^{-1}=q_{st}$, implying $(q_{ij}-q_{st})m_{is}m_{jt}=0$, thus ${\bf m}\in M(\mathcal{O}_{\bf q}(k^{n}))$. 

    Next we show that ${\rm Im} \Phi$ is normal. Suppose ${\bf m}\in {\rm Im} \Phi$, ${\bf n}\in M(\mathcal{O}_{\bf q}(k^{n}))$ and $\sigma\in {\rm Skel}({\bf n})$. By Lemma \ref{lemskel}(2), ${\bf r}_{\sigma}{\bf n}\in {\rm Im}\Phi$. Since ${\rm Im} \Phi$ is a group and ${\bf r}_{\sigma}^{-1}={\bf r}_{\sigma^{-1}}$, $({\bf r}_{\sigma}{\bf n}){\bf m}({\bf n}^{-1}{\bf r}_{\sigma}^{-1})={\bf r}_{\sigma}{\bf n}{\bf m}{\bf n}^{-1}{\bf r}_{\sigma^{-1}} \in {\rm Im}\Phi$. By Lemma \ref{lemcon}, $\sigma\in\mathcal{P}_{\bf q}$, thus $\sigma^{-1}\in\mathcal{P}_{\bf q}$. Then by Lemma \ref{lemmaPhi0}, one gets ${\bf nmn^{-1}}={\bf r}_{\sigma^{-1}}{\bf r}_\sigma{\bf nmn^{-1}}{\bf r}_{\sigma^{-1}}{\bf r}_\sigma\in {\rm Im}\Phi$, which completes the proof.
\end{proof}

\begin{theorem}\label{theoremmain}
    \[
      {\rm  Aut}_{\rm gr}(\mathcal{O}_{\bf q}(k^{n}))\cong (\prod_{1\le i\le m}{\rm GL}_{|B_i|}(k)) \rtimes \mathcal{P}_{\bf q}/\mathcal{I}_{\bf q} 
    \]
\end{theorem}
\begin{proof}
    By Lemma \ref{lemmaiq} and \ref{lemmaPhi2}, ${\rm Im} \Phi$ is a normal subgroup of $M(\mathcal{O}_{\bf q}(k^{n}))$, $\mathcal{I}_{\bf q}$ is a normal subgroup of $\mathcal{P}_{\bf q}$, thus $M(\mathcal{O}_{\bf q}(k^{n}))/{\rm Im} \Phi$ and $\mathcal{P}_{\bf q}/\mathcal{I}_{\bf q}$ are well-defined.
    Define a map 
    \[
        \Gamma: M(\mathcal{O}_{\bf q}(k^{n}))/{\rm Im} \Phi\rightarrow\mathcal{P}_{\bf q}/\mathcal{I}_{\bf q},$$ 
        $${\bf m}{\rm Im}\Phi\mapsto \pi\mathcal{I}_{\bf q},
    \]
    where $\pi\in {\rm Skel}({\bf m})$. We will show $\Gamma$ is a group isomorphism.
    
    Firstly, we claim that $\Gamma$ is well-defined. Suppose ${\bf m}{\rm Im}\Phi={\bf n}{\rm Im}\Phi$ for some ${\bf m,n}\in M(\mathcal{O}_{\bf q}(k^{n}))$. Then there exists ${\bf g}=(g_{ij})\in{\rm Im}\Phi$ such that ${\bf n}={\bf mg}$. Fix some $\pi\in {\rm Skel}({\bf m})$ and $\sigma\in {\rm Skel}({\bf mg})$. Note that ${\bf m}_{\pi(i)i}\ne 0$ and $({\bf mg})_{\sigma(i)i}\ne 0$ for $1\le i\le n$, therefore $({\bf mg})_{\sigma(i)i}=\sum\limits_{1\le l\le n}m_{\sigma(i)l}g_{li}\ne 0$ for all $1\le i\le n$. Since ${\bf g}\in {\rm Im}\ \Phi$, by Lemma \ref{lemmaimphi} one gets $g_{li}=0$ when $B(l)\ne B(i)$. Consequently,
    \[
        ({\bf mg})_{\sigma(i)i}  
        =\sum_{\substack{ 1\le l\le n,\\ B(l)=B(i)}}m_{\sigma(i)l}g_{li}\ne 0, \ 1\le i\le n.
    \]
    However, by Lemma \ref{lemskel}(1), $m_{\sigma(i)l}=m_{\pi(\pi^{-1}\sigma(i))l}\ne 0$ only if $B(\pi^{-1}\sigma(i))=B(l)$. Therefore, 
    \begin{equation*}
        ({\bf mg})_{\sigma(i)i}
        =\sum_{\substack{ 1\le l\le n,\\ B(l)=B(i)=B(\pi^{-1}\sigma(i))}}m_{\sigma(i)l}g_{li}\ne 0, \ 1\le i\le n.
    \end{equation*}
    Note that this equation is nonzero implies there is $1\le l\le n$ such that $B(l)=B(i)=B(\pi^{-1}\sigma(i))$, thus $B(i)= B(\pi^{-1}\sigma(i))$ for $1\le i\le n$. By Corollary \ref{corpi}, one gets $\pi^{-1}\sigma\in \mathcal{I}_{\bf q}$, thus 
    $\Gamma({\bf m}{\rm Im}\Phi)=\Gamma({\bf n}{\rm Im}\Phi)$. This shows the map $\Gamma$ is well-defined.

    \vskip5pt

    Next we show $\Gamma$ is a group isomorphism. Suppose that ${\bf m,n}\in M(\mathcal{O}_{\bf q}(k^n))$, $\pi\in {\rm Skel({\bf m})}$, $\sigma\in {\rm Skel({\bf n})}$ and $\tau\in {\rm Skel({\bf mn})}$. Then $\Gamma({\bf mn})=\tau\mathcal{I}_{\bf q}$, $\Gamma({\bf m})\Gamma({\bf n})=\pi\sigma\mathcal{I}_{\bf q}$. Note that $({\bf mn})_{\tau(i)i}\ne 0$ for $1\le i\le n$, one has
    \[
        ({\bf mn})_{\tau(i)i}=\sum_{\substack{ 1\le l\le n}}m_{\tau(i)l}n_{li}\ne 0, \ 1\le i\le n.
    \]
    However, by Lemma \ref{lemskel}(1), $m_{\tau(i)l}=m_{\pi(\pi^{-1}(\tau(i)))l}\ne 0$ only if $B(\pi^{-1}(\tau(i)))=B(l)$ and $n_{li}=n_{\sigma(\sigma^{-1}(l))i}\ne 0$ only if $B(\sigma^{-1}(l))=B(i)$. Thus
    \[
        ({\bf mn})_{\tau(i)i}=\sum_{\substack{ 1\le l\le n,\\ B(l)=B(\pi^{-1}\tau(i))=B(\sigma(i))}} m_{\tau(i)l}n_{li}\ne 0, \ 1\le i\le n.
    \]
    Note that this equation is nonzero implies there is $1\le l\le n$ such that $B(l)=B(\pi^{-1}\tau(i))=B(\sigma(i))$, thus $B(\pi^{-1}\tau(i))=B(\sigma(i))$ for $1\le i\le n$. And hence, as above, $\pi\sigma(B(i))= \tau(B(i))$, $1\le i\le n$ and $\Gamma({\bf m})\Gamma({\bf n})=\pi\sigma\mathcal{I}_{\bf q}=\tau\mathcal{I}_{\bf q}=\Gamma({\bf mn})$. Therefore, $\Gamma$ is a group homomorphism.

    For each $\pi \mathcal{I}_{\bf q}\in \mathcal{P}_{\bf q}/\mathcal{I}_{\bf q}$, $\Gamma({\bf r}_{\pi} {\rm Im}\Phi)=\pi \mathcal{I}_{\bf q}$, thus $\Gamma$ is surjective.    
    On the other hand, if $\Gamma({\bf m}{\rm Im}\Phi)=\pi \mathcal{I}_{\bf q}= \mathcal{I}_{\bf q}$, then $\pi\in \mathcal{I}_{\bf q}$. Note that $({\bf r}_{\pi})_{ij}=\delta_{i,\pi(j)}$ for $1\le i,j\le n$, thus $({\bf r}_{\pi})_{ij}=0$ when $B(i)\ne B(\pi(j))=B(j)$, and hence $({\bf r}_{\pi})_{B_i,B_j}=0$ for $1\le i\ne j\le m$. Therefore, ${\rm Det}({\bf r}_{\pi})=\prod\limits_{1\le i\le m}({\bf r}_{\pi})_{B_i,B_i}\ne 0$, which implies that $({\bf r}_{\pi})_{B_i,B_i}\in {\rm GL}_{|B_i|}(k)$ for $1\le i\le m$. By Lemma \ref{lemmaimphi}, we conclude that ${\bf r}_{\pi}\in{\rm Im}\Phi$. 

    Furthermore, by Lemma \ref{lemskel}(2), we know that ${\bf r}_{\pi}{\bf m}\in{\rm Im}\Phi$. Since ${\rm Im}\Phi$ is a group, ${\bf r}_{\pi}^{-1}$ is also in ${\rm Im}\Phi$. Thus ${\bf m} = {\bf r}_{\pi}^{-1} {\bf r}_{\pi}{\bf m}$ belongs to ${\rm Im}\Phi$, and one gets the injectivity of $\Gamma$. The proof is complete.
\end{proof}

In the commutative case, where ${\bf q}$ is an all-ones matrix, one has $m=1$ and $\mathcal{P}_{\bf q}=\mathcal{I}_{\bf q}=S_n$. This leads to the well-known result ${\rm Aut}_{\rm gr}(k[x_1,\cdots, x_n])\cong {\rm GL}_n(k)$.
On the other hand, when ${\bf q}$ has pairwise distinct rows, one has $m=n$, $B(i)=\{i\}$ for $1\le i\le n$ and $\mathcal{I}_{\bf q} = \{\Id\}$. Thus by Theorem \ref{theoremmain} one gets the following corollary:

\begin{corollary}
    If ${\bf q}$ has pairwise distinct rows, then ${\rm Aut}_{\rm gr}(\mathcal{O}_{\bf q}(k^{n}))\cong ({k^*})^n \rtimes \mathcal{P}_{\bf q}$.
\end{corollary}
\begin{examples}
    Suppose 
    ${\bf q} =
    \left(
        \begin{smallmatrix} 
            1 &1 &-1 &-1\\ 
            1 &1 &-1 &-1\\
            -1 &-1 &1 &1\\
            -1 &-1 &1 &1\\
        \end{smallmatrix}
    \right) $.
    Then one has $B_1 = \{1,2\},B_2 = \{3,4\}$ and ${\mathcal{I}_{\bf q}}=\{(1), (1,2), (3,4),(1,2)(3,4)\}$. Note that $\pi\in \mathcal{P}_{\bf q}$ is also a permutation on $\{B_1, B_2\}$, one gets 
    $$
    {\mathcal{P}_{\bf q}}=\{(1), (1,2), (3,4), (1,2)(3,4), (1,4)(2,3), (1,3)(2,4), (1,4,2,3), (1,3,2,4)\}.
    $$  
    By using {\rm GAP}\cite{GAP}, we find that $\mathcal{P}_{\bf q}\cong D_4$, $\mathcal{I}_{\bf q}\cong C_2^2$, where $C_2$ is the cyclic group of order $2$ and $D_4$ is the dihedral group of order $4$. Consequently, $\mathcal{P}_{\bf q}/\mathcal{I}_{\bf q}\cong C_2$. By Theorem $\ref{theoremmain}$, one gets ${\rm Aut_{gr}}(\mathcal{O}_{\bf q}(k^{4})) \cong ({\rm GL}_2(k)\times {\rm GL}_2(k)) \rtimes C_2.$
\end{examples}

\section*{Acknowledgments}
The author would like to thank the anonymous referees for their valuable comments and suggestions. The author would also like to thank Professor Pu Zhang for various helpful discussions. 

\bibliographystyle{amsalpha}

\end{document}